\documentclass{amsart}
\usepackage{amssymb}
\usepackage{amsfonts}
\usepackage{geometry}

\newtheorem{theorem}{Theorem}
\theoremstyle{plain}

\newtheorem{definition}[theorem]{Definition}

\newtheorem{lemma}[theorem]{Lemma}

\newtheorem{proposition}[theorem]{Proposition}
\newtheorem{remark}[theorem]{Remark}

\numberwithin{equation}{section}
\input{tcilatex}
\begin{document}
\title[$T1$ implies $Tp$]{$T1$ testing implies $Tp$ polynomial testing:
optimal cancellation conditions for $CZO$'s}
\author[E.T. Sawyer]{Eric T. Sawyer}
\address{ Department of Mathematics \& Statistics, McMaster University, 1280
Main Street West, Hamilton, Ontario, Canada L8S 4K1 }

\begin{abstract}
This paper is the third in an investigation begun in \texttt{arXiv:1906.05602%
} and \texttt{arXiv:1907.07571} of extending the $T1\,$theorem of David and
Journ\'{e}, and optimal cancellation conditions, to more general weight
pairs. The main result here is that the familiar $T1$ testing conditions
over indicators of cubes, together with the one-tailed $\mathcal{A}_{2}$
conditions, imply polynomial testing. Analogous results for fractional
singular integrals hold as well.

Applications include a $T1$ theorem for $\alpha $-fractional $CZO$'s $%
T^{\alpha }$ in the case of doubling measures when one of the weights is $%
A_{\infty }$ (and $T^{0}$ is bounded on unweighted $L^{2}\left( \mathbb{R}%
^{n}\right) $ when $\alpha =0$), and then to optimal cancellation conditions
for such $CZO$'s in similar situations.
\end{abstract}

\dedicatory{In memory of Professor Elias M. Stein.}
\email{sawyer@mcmaster.ca}
\maketitle
\tableofcontents

\section{Introduction and definitions}

In Theorem \ref{Tp control by T1} below, we show that for fractional Calder%
\'{o}n-Zygmund operators, the $\kappa $-Cube Testing conditions over
polynomials of degree less than $\kappa $ times indicators of cubes, are
`essentially' controlled by the familiar $1$-Cube Testing conditions over
indicators of cubes. This is then applied in Theorem \ref{Stein extension}
at the end of the paper, to obtain an extension of Stein's characterization 
\cite[Theorem 4 page 306]{Ste2}\footnote{%
`actually a rather direct consequence of the $T1$ theorem' in the words of
Stein \cite[page 306]{Ste2}.} of optimal cancellation conditions for Calder%
\'{o}n-Zygmund operators via the $T1$ theorem of David and Journ\'{e}. The
extension is to more general pairs of doubling measures, with one weight in $%
A_{\infty }$, in place of Lebesgue measure. We now recall the definitions
needed to formulate and prove these theorems.

Let $\sigma $ and $\omega $ be locally finite positive Borel measures on $%
\mathbb{R}^{n}$, and denote by $\mathcal{P}^{n}$ the collection of all cubes
in $\mathbb{R}^{n}$ with sides parallel to the coordinate axes. For $0\leq
\alpha <n$, the classical $\alpha $-fractional Muckenhoupt condition for the
measure pair $\left( \sigma ,\omega \right) $ is given by%
\begin{equation}
A_{2}^{\alpha }\left( \sigma ,\omega \right) \equiv \sup_{Q\in \mathcal{P}%
^{n}}\frac{\left\vert Q\right\vert _{\sigma }}{\left\vert Q\right\vert ^{1-%
\frac{\alpha }{n}}}\frac{\left\vert Q\right\vert _{\omega }}{\left\vert
Q\right\vert ^{1-\frac{\alpha }{n}}}<\infty ,  \label{frac Muck}
\end{equation}%
and the one-tailed conditions by%
\begin{eqnarray}
\mathcal{A}_{2}^{\alpha }\left( \sigma ,\omega \right) &\equiv &\sup_{Q\in 
\mathcal{Q}^{n}}\mathcal{P}^{\alpha }\left( Q,\sigma \right) \frac{%
\left\vert Q\right\vert _{\omega }}{\left\vert Q\right\vert ^{1-\frac{\alpha 
}{n}}}<\infty ,  \label{one-sided} \\
\mathcal{A}_{2}^{\alpha ,\ast }\left( \sigma ,\omega \right) &\equiv
&\sup_{Q\in \mathcal{Q}^{n}}\frac{\left\vert Q\right\vert _{\sigma }}{%
\left\vert Q\right\vert ^{1-\frac{\alpha }{n}}}\mathcal{P}^{\alpha }\left(
Q,\omega \right) <\infty ,  \notag
\end{eqnarray}%
where the reproducing Poisson integral $\mathcal{P}^{\alpha }$ is given by%
\begin{equation*}
\mathcal{P}^{\alpha }\left( Q,\mu \right) \equiv \int_{\mathbb{R}^{n}}\left( 
\frac{\left\vert Q\right\vert ^{\frac{1}{n}}}{\left( \left\vert Q\right\vert
^{\frac{1}{n}}+\left\vert x-x_{Q}\right\vert \right) ^{2}}\right) ^{n-\alpha
}d\mu \left( x\right) .
\end{equation*}%
The measure $\sigma $ is said to be \emph{doubling} if there is a positive
constant $C_{\limfunc{doub}}$, called the doubling constant, such that%
\begin{equation}
\left\vert 2Q\right\vert _{\mu }\leq C_{\limfunc{doub}}\left\vert
Q\right\vert _{\mu }\ ,\ \ \ \ \ \text{for all cubes }Q\in \mathcal{P}^{n}.
\label{def rev doub}
\end{equation}%
The absolutely continuous measure $d\omega \left( x\right) =w\left( x\right)
dx$ is said to be an $A_{\infty }$ weight if there are constants $%
0<\varepsilon ,\eta <1$, called $A_{\infty }$ parameters, such that%
\begin{equation*}
\frac{\left\vert E\right\vert _{\omega }}{\left\vert Q\right\vert _{\omega }}%
<\eta \text{ whenever }E\text{ compact }\subset Q\text{ a cube with }\frac{%
\left\vert E\right\vert }{\left\vert Q\right\vert }<\varepsilon .
\end{equation*}

Let $0\leq \alpha <n$. For $\kappa _{1},\kappa _{2}\in \mathbb{N}$ and $%
\delta >0$, we say that $K^{\alpha }\left( x,y\right) $ is a standard $%
\left( \kappa _{1}+\delta ,\kappa _{2}+\delta \right) $-smooth $\alpha $%
-fractional kernel if for $x\neq y$, and with $\nabla _{1}$ denoting
gradient in the first variable, and $\nabla _{2}$ denoting gradient in the
second variable, 
\begin{eqnarray}
&&\left\vert \nabla _{1}^{j}K^{\alpha }\left( x,y\right) \right\vert \leq
C_{CZ}\left\vert x-y\right\vert ^{\alpha -j-n-1},\ \ \ \ \ 0\leq j\leq
\kappa _{1},  \label{sizeandsmoothness'} \\
&&\left\vert \nabla _{1}^{\kappa }K^{\alpha }\left( x,y\right) -\nabla
_{1}^{\kappa }K^{\alpha }\left( x^{\prime },y\right) \right\vert \leq
C_{CZ}\left( \frac{\left\vert x-x^{\prime }\right\vert }{\left\vert
x-y\right\vert }\right) ^{\delta }\left\vert x-y\right\vert ^{\alpha -\kappa
_{1}-n-1},\ \ \ \ \ \frac{\left\vert x-x^{\prime }\right\vert }{\left\vert
x-y\right\vert }\leq \frac{1}{2},  \notag
\end{eqnarray}%
and where the same inequalities hold for the adjoint kernel $K^{\alpha ,\ast
}\left( x,y\right) \equiv K^{\alpha }\left( y,x\right) $, in which $x$ and $%
y $ are interchanged, and where $\kappa _{1}$ is replaced by $\kappa _{2}$,
and $\nabla _{1}$\ by $\nabla _{2}$. We also consider vector kernels $%
K^{\alpha }=\left( K_{j}^{\alpha }\right) $ where each $K_{j}^{\alpha }$ is
as above, often without explicit mention. This includes for example the
vector Riesz transform in higher dimensions.

Given a standard $\alpha $-fractional $CZ$ kernel $K^{\alpha }$, we consider
truncated kernels $K_{\delta ,R}^{\alpha }\left( x,y\right) =\eta _{\delta
,R}^{\alpha }\left( \left\vert x-y\right\vert \right) K^{\alpha }\left(
x,y\right) $ which uniformly satisfy (\ref{sizeandsmoothness'}). Then the
truncated operator $T_{\delta ,R}^{\alpha }$ with kernel $K_{\delta
,R}^{\alpha }$ is pointwise well-defined, and we will refer to the pair $%
T^{\alpha }=\left( K^{\alpha },\left\{ \eta _{\delta ,R}^{\alpha }\right\}
_{0<\delta <R<\infty }\right) $ as an $\alpha $-fractional singular integral
operator.

\begin{definition}
We say that an $\alpha $-fractional singular integral operator $T^{\alpha
}=\left( K^{\alpha },\left\{ \eta _{\delta ,R}^{\alpha }\right\} _{0<\delta
<R<\infty }\right) $ satisfies the norm inequality%
\begin{equation}
\left\Vert T_{\sigma }^{\alpha }f\right\Vert _{L^{2}\left( \omega \right)
}\leq \mathfrak{N}_{T_{\sigma }^{\alpha }}\left( \sigma ,\omega \right)
\left\Vert f\right\Vert _{L^{2}\left( \sigma \right) },\ \ \ \ \ f\in
L^{2}\left( \sigma \right) ,  \label{two weight'}
\end{equation}%
provided%
\begin{equation*}
\left\Vert T_{\sigma ,\delta ,R}^{\alpha }f\right\Vert _{L^{2}\left( \omega
\right) }\leq \mathfrak{N}_{T_{\sigma }^{\alpha }}\left( \sigma ,\omega
\right) \left\Vert f\right\Vert _{L^{2}\left( \sigma \right) },\ \ \ \ \
f\in L^{2}\left( \sigma \right) ,0<\delta <R<\infty .
\end{equation*}
\end{definition}

In the presence of the classical Muckenhoupt condition $A_{2}^{\alpha }$,
the norm inequality (\ref{two weight'}) is essentially independent of the
choice of truncations used, including \emph{nonsmooth} truncations as well -
see e.g. \cite{LaSaShUr3}.

The $\kappa $\emph{-cube testing conditions} associated with an $\alpha $%
-fractional singular integral operator $T^{\alpha }$ introduced by Rahm,
Sawyer and Wick in \cite{RaSaWi} are given, with a slight modification, by%
\begin{eqnarray}
\left( \mathfrak{T}_{T^{\alpha }}^{\left( \kappa \right) }\left( \sigma
,\omega \right) \right) ^{2} &\equiv &\sup_{Q\in \mathcal{P}^{n}}\max_{0\leq
\left\vert \beta \right\vert <\kappa }\frac{1}{\left\vert Q\right\vert
_{\sigma }}\int_{Q}\left\vert T_{\sigma }^{\alpha }\left( \mathbf{1}%
_{Q}m_{Q}^{\beta }\right) \right\vert ^{2}\omega <\infty ,
\label{def Kappa polynomial} \\
\left( \mathfrak{T}_{\left( T^{\alpha }\right) ^{\ast }}^{\left( \kappa
\right) }\left( \omega ,\sigma \right) \right) ^{2} &\equiv &\sup_{Q\in 
\mathcal{P}^{n}}\max_{0\leq \left\vert \beta \right\vert <\kappa }\frac{1}{%
\left\vert Q\right\vert _{\omega }}\int_{Q}\left\vert \left( T^{\alpha
}\right) _{\omega }^{\ast }\left( \mathbf{1}_{Q}m_{Q}^{\beta }\right)
\right\vert ^{2}\sigma <\infty ,  \notag
\end{eqnarray}%
with $m_{Q}^{\beta }\left( x\right) \equiv \left( \frac{x-c_{Q}}{\ell \left(
Q\right) }\right) ^{\beta }$ for any cube $Q$ and multiindex $\beta $, where 
$c_{Q}$ is the center of the cube $Q$, and where as usual we interpret the
right hand sides as holding uniformly over all sufficiently smooth
truncations of $T^{\alpha }$. The more familiar cube testing conditions, as
found in $T1$ theorems, are the case $\kappa =1$ of (\ref{def Kappa
polynomial}) and $m_{Q}^{\beta }=1$.

We also use the larger \emph{full }$\kappa $\emph{-cube testing conditions}
in which the integrals over $Q$ are extended to the whole space $\mathbb{R}%
^{n}$:%
\begin{eqnarray*}
\left( \mathfrak{FT}_{T^{\alpha }}^{\left( \kappa \right) }\left( \sigma
,\omega \right) \right) ^{2} &\equiv &\sup_{Q\in \mathcal{P}^{n}}\max_{0\leq
\left\vert \beta \right\vert <\kappa }\frac{1}{\left\vert Q\right\vert
_{\sigma }}\int_{\mathbb{R}^{n}}\left\vert T_{\sigma }^{\alpha }\left( 
\mathbf{1}_{Q}m_{Q}^{\beta }\right) \right\vert ^{2}\omega <\infty , \\
\left( \mathfrak{FT}_{\left( T^{\alpha }\right) ^{\ast }}^{\left( \kappa
\right) }\left( \omega ,\sigma \right) \right) ^{2} &\equiv &\sup_{Q\in 
\mathcal{P}^{n}}\max_{0\leq \left\vert \beta \right\vert <\kappa }\frac{1}{%
\left\vert Q\right\vert _{\omega }}\int_{\mathbb{R}^{n}}\left\vert \left(
T^{\alpha }\right) _{\omega }^{\ast }\left( \mathbf{1}_{Q}m_{Q}^{\beta
}\right) \right\vert ^{2}\sigma <\infty .
\end{eqnarray*}

Finally, as in \cite{SaShUr7}, an $\alpha $-fractional vector Calder\'{o}%
n-Zygmund kernel $K^{\alpha }=\left( K_{j}^{\alpha }\right) $ is said to be 
\emph{elliptic} if there is $c>0$ such that for each unit vector $\mathbf{u}%
\in \mathbb{R}^{n}$ there is $j$ satisfying%
\begin{equation*}
\left\vert K_{j}^{\alpha }\left( x,x+tu\right) \right\vert \geq ct^{\alpha
-n},\ \ \ \ \ \text{for all }t>0;
\end{equation*}%
and $K^{\alpha }=\left( K_{j}^{\alpha }\right) $ is said to be \emph{%
strongly elliptic} if for each $m\in \left\{ 1,-1\right\} ^{n}$, there is a
sequence of coefficients $\left\{ \lambda _{j}^{m}\right\} _{j=1}^{J}$ such
that%
\begin{equation}
\left\vert \sum_{j=1}^{J}\lambda _{j}^{m}K_{j}^{\alpha }\left( x,x+t\mathbf{u%
}\right) \right\vert \geq ct^{\alpha -n},\ \ \ \ \ t\in \mathbb{R}.
\label{Ktalpha strong}
\end{equation}%
holds for \emph{all} unit vectors $\mathbf{u}$ in the $n$-ant 
\begin{equation*}
V_{m}=\left\{ x\in \mathbb{R}^{n}:m_{i}x_{i}>0\text{ for }1\leq i\leq
n\right\} ,\ \ \ \ \ m\in \left\{ 1,-1\right\} ^{n}.
\end{equation*}%
For example, the vector Riesz transform kernel is strongly elliptic (\cite%
{SaShUr7}).

\subsection{Controlling polynomial testing conditions - main theorems}

We begin in dimension $n=1$ with the elementary formula for recovering a
linear function from indicators of intervals,%
\begin{equation}
\mathbf{1}_{\left[ a,b\right) }\left( y\right) \left( \frac{y-a}{b-a}\right)
=\int_{a}^{b}\mathbf{1}_{\left[ r,b\right) }\left( y\right) \frac{dr}{b-a},\
\ \ \ \ \text{for all }y\in \mathbb{R},  \label{elem form}
\end{equation}%
from which we conclude that for any locally finite positive Borel measure $%
\sigma $, and any operator $T$ bounded from $L^{2}\left( \sigma \right) $ to 
$L^{2}\left( \omega \right) $,%
\begin{equation*}
T_{\sigma }\left( \mathbf{1}_{\left[ a,b\right) }\left( y\right) \left( 
\frac{y-a}{b-a}\right) \right) \left( x\right) =T_{\sigma }\left(
\int_{a}^{b}\mathbf{1}_{\left[ r,b\right) }\left( y\right) \frac{dr}{b-a}%
\right) \left( x\right) =\int_{a}^{b}\left( T_{\sigma }\mathbf{1}_{\left[
r,b\right) }\right) \left( x\right) \frac{dr}{b-a}.
\end{equation*}%
We then use the testing estimate $\left\Vert T_{\sigma }\mathbf{1}_{\left[
r,b\right) }\right\Vert _{L^{2}\left( \omega \right) }^{2}\leq \left( 
\mathfrak{FT}_{T}\right) ^{2}\left\vert \left[ r,b\right) \right\vert
_{\sigma }$, together with Minkowski's inequality, to obtain 
\begin{eqnarray*}
&&\left\Vert T_{\sigma }\left[ \mathbf{1}_{\left[ a,b\right) }\left(
y\right) \left( \frac{y-a}{b-a}\right) \right] \right\Vert _{L^{2}\left(
\omega \right) }=\left\Vert T_{\sigma }\left[ \int_{a}^{b}\mathbf{1}_{\left[
r,b\right) }\left( y\right) \frac{dr}{b-a}\right] \right\Vert _{L^{2}\left(
\omega \right) } \\
&\leq &\int_{a}^{b}\left\Vert T_{\sigma }\left[ \mathbf{1}_{\left[
r,b\right) }\left( y\right) \right] \right\Vert _{L^{2}\left( \omega \right)
}\frac{dr}{b-a}\leq \int_{a}^{b}\mathfrak{FT}_{T}\sqrt{\left\vert \left[
r,b\right) \right\vert _{\sigma }}\frac{dr}{b-a} \\
&\leq &\mathfrak{FT}_{T}\sqrt{\int_{a}^{b}\left\vert \left[ r,b\right)
\right\vert _{\sigma }\frac{dr}{b-a}}=\mathfrak{FT}_{T}\sqrt{%
\int_{a}^{b}\left( \int_{\left[ r,b\right) }d\sigma \left( y\right) \right) 
\frac{dr}{b-a}} \\
&=&\mathfrak{FT}_{T}\sqrt{\int_{\left[ a,b\right) }\left( \int_{a}^{y}\frac{%
dr}{b-a}\right) d\sigma \left( y\right) }=\mathfrak{FT}_{T}\sqrt{\int_{\left[
a,b\right) }\frac{y-a}{b-a}d\sigma \left( y\right) }\leq \mathfrak{FT}_{T}%
\sqrt{\left\vert \left[ a,b\right) \right\vert _{\sigma }}\ ,
\end{eqnarray*}%
and hence $\mathfrak{FT}_{T}^{\left( 1\right) }\leq \mathfrak{FT}%
_{T}^{\left( 0\right) }\equiv \mathfrak{FT}_{T}$. Similarly, the identity%
\begin{equation*}
\mathbf{1}_{\left[ a,b\right) }\left( y\right) \left( \frac{y-a}{b-a}\right)
^{2}=\int_{a}^{b}1_{\left[ r,b\right) }\left( y\right) 2\left( \frac{y-r}{b-a%
}\right) \frac{dr}{b-a},\ \ \ \ \ \text{for all }y\in \mathbb{R},
\end{equation*}%
shows that%
\begin{eqnarray*}
&&\left\Vert T\left[ \mathbf{1}_{\left[ a,b\right) }\left( y\right) \left( 
\frac{y-a}{b-a}\right) ^{2}\right] \right\Vert _{L^{2}\left( \omega \right)
}=\left\Vert T\left[ \int_{a}^{b}\mathbf{1}_{\left[ r,b\right) }\left(
y\right) 2\left( \frac{y-r}{b-a}\right) dr\right] \right\Vert _{L^{2}\left(
\omega \right) } \\
&\leq &2\int_{a}^{b}\left\Vert T\left[ \mathbf{1}_{\left[ r,b\right) }\left(
y\right) \left( \frac{y-r}{b-a}\right) \right] \right\Vert _{L^{2}\left(
\omega \right) }\frac{dr}{b-a}\leq 2\mathfrak{FT}_{T}^{\left( 1\right) }%
\sqrt{\left\vert \left[ a,b\right) \right\vert _{\sigma }},
\end{eqnarray*}%
and hence $\mathfrak{FT}_{T}^{\left( 2\right) }\leq 2\mathfrak{FT}%
_{T}^{\left( 1\right) }$. Continuing in this manner we obtain%
\begin{equation*}
\mathfrak{FT}_{T}^{\left( \kappa \right) }\leq \kappa \ \mathfrak{FT}%
_{T}^{\left( \kappa -1\right) },\ \ \ \ \ \text{for all }\kappa \geq 1,
\end{equation*}%
which when iterated gives%
\begin{equation*}
\mathfrak{FT}_{T}^{\left( \kappa \right) }\left( \sigma ,\omega \right) \leq
\kappa !\mathfrak{FT}_{T}\left( \sigma ,\omega \right) .
\end{equation*}

By a result of Hyt\"{o}nen \cite{Hyt2}, see also \cite{SaShUr12} for the
straightforward extension to fractional singular integrals, the full testing
constant $\mathfrak{FT}_{T}\left( \sigma ,\omega \right) $ in dimension $n=1$%
, is controlled by the usual testing constant $\mathfrak{T}_{T}\left( \sigma
,\omega \right) $ and the one-tailed Muckenhoupt condition $\mathcal{A}%
_{2}^{\alpha }$ . Thus we have proved the following lemma for the case when $%
T=T^{\alpha }$ is a fractional CZ operator.

\begin{lemma}
Suppose that $\sigma $ and $\omega $ are locally finite positive Borel
measures on $\mathbb{R}$ and $\kappa \in \mathbb{N}$. If $T^{\alpha }$ is a
bounded $\alpha $-fractional CZ operator from $L^{2}\left( \sigma \right) $
to $L^{2}\left( \omega \right) $, then we have 
\begin{equation*}
\mathfrak{T}_{T^{\alpha }}^{\left( \kappa \right) }\left( \sigma ,\omega
\right) \leq \kappa !\mathfrak{T}_{T^{\alpha }}\left( \sigma ,\omega \right)
+C_{\kappa }\mathcal{A}_{2}^{\alpha }\left( \sigma ,\omega \right) \ ,\ \ \
\ \ \kappa \geq 1,
\end{equation*}%
where the constant $C_{\kappa }$ depends on the kernel constant $C_{CZ}$ in (%
\ref{sizeandsmoothness'}), but is independent of the operator norm $%
\mathfrak{N}_{T^{\alpha }}\left( \sigma ,\omega \right) $.
\end{lemma}

The higher dimensional version of this lemma will include a small multiple
of the operator norm $\mathfrak{N}_{T}\left( \sigma ,\omega \right) $ on the
right hand side, since we no longer have available an analogue of Hyt\"{o}%
nen's result. Nevertheless, for doubling measures, the two testing
conditions are equivalent in the presence of one-tailed Muckenhoupt
conditions (\ref{one-sided}) in all dimensions, and so we will be able to
prove a $T1$ theorem in higher dimensions in certain cases.

\begin{theorem}
\label{Tp control by T1}Suppose that $\sigma $ and $\omega $ are locally
finite positive Borel measures on $\mathbb{R}^{n}$, and let $\kappa \in 
\mathbb{N}$. If $T$ is a bounded operator from $L^{2}\left( \sigma \right) $
to $L^{2}\left( \omega \right) $, then\ for every $0<\varepsilon <1$, there
is a positive constant $C\left( \kappa ,\varepsilon \right) $ such that 
\begin{equation*}
\mathfrak{FT}_{T}^{\left( \kappa \right) }\left( \sigma ,\omega \right) \leq
C\left( \kappa ,\varepsilon \right) \mathfrak{FT}_{T}\left( \sigma ,\omega
\right) +\varepsilon \mathfrak{N}_{T}\left( \sigma ,\omega \right) \ ,\ \ \
\ \ \kappa \geq 1,
\end{equation*}%
and where the constants $C\left( \kappa ,\varepsilon \right) $ depend only
on $\kappa $ and $\varepsilon $, and not on the operator norm $\mathfrak{N}%
_{T}\left( \sigma ,\omega \right) $.
\end{theorem}

\begin{proof}
We begin with the following geometric observation, similar to a construction
used in the recursive control of the nearby form in \cite{SaShUr12}. Let $R=%
\left[ 0,1\right) ^{n-1}\times \left[ 0,t\right) $ be a rectangle in $%
\mathbb{R}^{n}$ with $0<t<1$. Then given $0<\varepsilon <1$, there is a
positive integer $m\in \mathbb{N}$ and a dyadic number $t^{\ast }\equiv 
\frac{b}{2^{m}}$ with $0\leq b<2^{m}$, so that%
\begin{eqnarray}
R &=&E\overset{\cdot }{\cup }\left\{ \overset{\cdot }{\dbigcup }%
_{i=1}^{B}K_{i}\right\} ;  \label{decomp} \\
E &=&\left[ 0,1\right) ^{n-1}\times \left[ t^{\ast },t\right) \text{ with }%
\left\vert t-t^{\ast }\right\vert <\varepsilon ,  \notag \\
B &\leq &2^{nm-n-m+2},  \notag
\end{eqnarray}%
and where the $K_{i}$ are pairwise disjoint cubes inside $R$. To see (\ref%
{decomp}) we choose $m\in \mathbb{N}$ so that $\frac{1}{2^{m}}<\varepsilon $
and then let $b\in \mathbb{N}$ satisfy $2^{m}t-1\leq b<2^{m}t$. Then with $%
t^{\ast }=\frac{b}{2^{m}}$ we have $\left\vert t-t^{\ast }\right\vert <\frac{%
1}{2^{m}}<\varepsilon $. Now expand $t^{\ast }$ in binary form,%
\begin{equation*}
t^{\ast }=b_{1}\frac{1}{2}+b_{2}\frac{1}{4}+...+b_{m-1}\frac{1}{2^{m-1}},\ \
\ \ \ b_{k}\in \left\{ 0,1\right\} .
\end{equation*}%
Then for each $k$ with $b_{k}=1$ we decompose the rectangle 
\begin{equation*}
R_{k}\equiv \left[ 0,1\right) ^{n-1}\times \left[ b_{1}\frac{1}{2}+b_{2}%
\frac{1}{4}+...+b_{k-1}\frac{1}{2^{k-1}},b_{1}\frac{1}{2}+b_{2}\frac{1}{4}%
+...+b_{k-1}\frac{1}{2^{k-1}}+\frac{1}{2^{k}}\right)
\end{equation*}%
into $2^{\left( n-1\right) k}$ pairwise disjoint dyadic cubes of side length 
$\frac{1}{2^{k}}$. Then we take the collection of all such cubes, noting
that the number $B$ of such cubes is at most 
\begin{equation*}
\sum_{k=1}^{m-1}2^{\left( n-1\right) k}\leq 2\cdot 2^{\left( n-1\right)
\left( m-1\right) }=2^{nm-n-m+2},
\end{equation*}%
and label them as $\left\{ K_{i}\right\} _{i=1}^{B}$ with $B\leq
2^{nm-n-m+2} $. Finally we note that 
\begin{equation*}
\overset{\cdot }{\dbigcup }_{i=1}^{B}K_{i}=\overset{\cdot }{\dbigcup }_{k:\
b_{k}=1}R_{k}=\left[ 0,1\right) ^{n-1}\times \left[ 0,t^{\ast }\right) .
\end{equation*}%
This completes the proof of (\ref{decomp}). Note that we may arrange to have 
$m\approx \ln \frac{1}{\varepsilon }$.

We also have the same result for the complementary rectangle $R=\left[
0,1\right) ^{n-1}\times \left[ r,1\right) $ by simply reflecting about the
plane $y_{n}=\frac{1}{2}$ and taking $r=1-t$. It is in this complementary
form that we will use (\ref{decomp}).

Again we start by considering the full testing condition $\mathfrak{FT}%
_{T}^{1}$ over linear functions, and we begin by estimating%
\begin{equation*}
\left\Vert T_{\sigma }\left[ \mathbf{1}_{Q}\left( y\right) y_{j}\right]
\right\Vert _{L^{2}\left( \omega \right) }^{2},\ \ \ \ \ Q\in \mathcal{P}%
^{n},1\leq j\leq n.
\end{equation*}%
In order to reduce notational clutter in appealing to the complementary form
of the geometric observation above, we will suppose - without loss of
generality - that $Q=\left[ 0,1\right) ^{n}$ is the unit cube in $\mathbb{R}%
^{n}$, and that $j=n$. Then we have%
\begin{equation*}
\mathbf{1}_{\left[ 0,1\right) ^{n}}\left( y\right) y_{n}=\int_{0}^{1}\mathbf{%
1}_{\left[ 0,1\right) ^{n-1}\times \left[ r,1\right) }\left( y\right) dr,\ \
\ \ \ \text{for all }y\in \mathbb{R}^{n},
\end{equation*}%
and%
\begin{equation*}
T_{\sigma }\left( \mathbf{1}_{\left[ 0,1\right) ^{n}}\left( y\right)
y_{n}\right) \left( x\right) =T_{\sigma }\left( \int_{0}^{1}\mathbf{1}_{%
\left[ 0,1\right) ^{n-1}\times \left[ r,1\right) }\left( y\right) dr\right)
\left( x\right) =\int_{0}^{1}\left( T_{\sigma }\mathbf{1}_{\left[ 0,1\right)
^{n-1}\times \left[ r,1\right) }\right) \left( x\right) dr.
\end{equation*}%
The norm estimate is complicated by the lack of Hyt\"{o}nen's result in
higher dimensions, and we compensate by using the complementary form of the
geometric observation (\ref{decomp}), together with a simple probability
argument. Let $\left[ r,1\right) =\left[ r,r^{\ast }\right) \overset{\cdot }{%
\cup }\left[ r^{\ast },1\right) $ and write%
\begin{eqnarray*}
&&\left\Vert T_{\sigma }\mathbf{1}_{\left[ 0,1\right) ^{n-1}\times \left[
r,1\right) }\right\Vert _{L^{2}\left( \omega \right) }^{2}=\int \left\vert
T_{\sigma }\left\{ \mathbf{1}_{\left[ 0,1\right) ^{n-1}\times \left[
r,r^{\ast }\right) }+\mathbf{1}_{\left[ 0,1\right) ^{n-1}\times \left[
r^{\ast },1\right) }\right\} \left( x\right) \right\vert ^{2}d\omega \left(
x\right) \\
&=&\int \left\vert T_{\sigma }\left\{ \mathbf{1}_{\left[ 0,1\right)
^{n-1}\times \left[ r,r^{\ast }\right) }+\sum_{i=1}^{B}\mathbf{1}%
_{K_{i}}\right\} \left( x\right) \right\vert ^{2}d\omega \left( x\right) \\
&\lesssim &\int \left\vert T_{\sigma }\mathbf{1}_{\left[ 0,1\right)
^{n-1}\times \left[ r,r^{\ast }\right) }\left( x\right) \right\vert
^{2}d\omega \left( x\right) +\sum_{i=1}^{B}\int \left\vert T_{\sigma }%
\mathbf{1}_{K_{i}}\left( x\right) \right\vert ^{2}d\omega \left( x\right) \\
&\leq &\int \left\vert T_{\sigma }\mathbf{1}_{\left[ 0,1\right) ^{n-1}\times %
\left[ r,r^{\ast }\right) }\left( x\right) \right\vert ^{2}d\omega \left(
x\right) +\left( \mathfrak{FT}_{T}\right) ^{2}\sum_{i=1}^{B}\left\vert
K_{i}\right\vert _{\sigma }.
\end{eqnarray*}%
First, we apply a simple probability argument to the integral over $r$ of
the last integral above by pigeonholing the values taken by $r^{\ast }\in
\left\{ \frac{b}{2^{m}}\right\} _{0\leq b<2^{m}}$: 
\begin{eqnarray*}
&&\int_{0}^{1}\int \left\vert T_{\sigma }\mathbf{1}_{\left[ 0,1\right)
^{n-1}\times \left[ r,r^{\ast }\right) }\left( x\right) \right\vert
^{2}d\omega \left( x\right) dr\leq \mathfrak{N}_{T}\left( \sigma ,\omega
\right) ^{2}\int_{0}^{1}\left\{ \int_{\left[ 0,1\right) ^{n-1}\times \left[
r,r^{\ast }\right) }d\sigma \right\} dr \\
&=&\mathfrak{N}_{T}\left( \sigma ,\omega \right) ^{2}\sum_{0<b\leq
2^{m}}\int_{\frac{b-1}{2^{m}}}^{\frac{b}{2^{m}}}\left\{ \int_{\left[
0,1\right) ^{n-1}\times \left[ r,\frac{b}{2^{m}}\right) }d\sigma \right\}
dr\leq \mathfrak{N}_{T}\left( \sigma ,\omega \right) ^{2}\int_{\left[
0,1\right) ^{n}}\left\{ \int_{y_{n}-\varepsilon }^{y_{n}}dr\right\} d\sigma
\left( y_{1},...,y_{n}\right) \\
&\leq &\mathfrak{N}_{T}\left( \sigma ,\omega \right) ^{2}\int_{\left[
0,1\right) ^{n}}\varepsilon d\sigma \left( y_{1},...,y_{n}\right)
=\varepsilon \mathfrak{N}_{T}\left( \sigma ,\omega \right) ^{2}\left\vert %
\left[ 0,1\right) ^{n}\right\vert _{\sigma }\ ,
\end{eqnarray*}%
since $\frac{b-1}{2^{m}}\leq r\leq y_{n}<\frac{b}{2^{m}}$ implies $%
y_{n}-\varepsilon <y_{n}-\frac{1}{2^{m}}\leq r\leq y_{n}$.

Combining estimates, and setting $R_{r}\equiv \left[ 0,1\right) ^{n-1}\times %
\left[ r,1\right) $ for convenience, we obtain%
\begin{eqnarray*}
&&\left\Vert T_{\sigma }\left[ \mathbf{1}_{R_{r}}\left( y\right) y_{n}\right]
\right\Vert _{L^{2}\left( \omega \right) }=\left\Vert T_{\sigma }\left[
\int_{0}^{1}\mathbf{1}_{R_{r}}\left( y\right) dr\right] \right\Vert
_{L^{2}\left( \omega \right) } \\
&\leq &\int_{0}^{1}\left\Vert T_{\sigma }\left[ \mathbf{1}_{R_{r}}\left(
y\right) \right] \right\Vert _{L^{2}\left( \omega \right) }dr\leq \mathfrak{%
FT}_{T}\left( \sigma ,\omega \right) \int_{0}^{1}\sqrt{\left\vert
R_{r}\right\vert _{\sigma }}dr+\varepsilon \mathfrak{N}_{T}\left( \sigma
,\omega \right) \left\vert \left[ 0,1\right) ^{n}\right\vert _{\sigma }\ ,
\end{eqnarray*}%
where%
\begin{eqnarray*}
\int_{0}^{1}\sqrt{\left\vert R_{r}\right\vert _{\sigma }}dr &\leq &\sqrt{%
\int_{0}^{1}\left\vert \left[ r,b\right) \right\vert _{\sigma }\frac{dr}{b-a}%
}=\sqrt{\int_{0}^{1}\int_{\left[ 0,1\right) ^{n-1}\times \left[ r,1\right)
}d\sigma \left( y\right) dr} \\
&=&\sqrt{\int_{\left[ 0,1\right) ^{n}}\int_{\left[ 0,y_{n}\right) }drd\sigma
\left( y\right) }=\sqrt{\int_{\left[ 0,1\right) ^{n}}y_{n}d\sigma \left(
y\right) }\ .
\end{eqnarray*}%
Noting that $\sqrt{\int_{\left[ 0,1\right) ^{n}}y_{n}d\sigma \left( y\right) 
}\leq \left\vert \left[ 0,1\right) ^{n}\right\vert _{\sigma }$, that the
same estimates hold for $y_{j}$ in place of $y_{n}$, and finally that there
are appropriate analogues of these estimates for all cubes $Q\in \mathcal{P}%
^{n}$ in place of $\left[ 0,1\right) ^{n}$, we see that 
\begin{equation*}
\mathfrak{FT}_{T}^{\left( 1\right) }\left( \sigma ,\omega \right) \leq
C_{m,0}\mathfrak{FT}\left( \sigma ,\omega \right) +\varepsilon \mathfrak{N}%
_{T}\left( \sigma ,\omega \right) \mathfrak{.}
\end{equation*}

Similarly, for each $i<n$ we can consider the monomial $y_{i}y_{n}$, and
obtain from the above argument with $y_{i}$ included in the integrand, that%
\begin{equation}
\left\Vert T_{\sigma }\left[ \mathbf{1}_{R_{r}}\left( y\right) y_{i}y_{n}%
\right] \right\Vert _{L^{2}\left( \omega \right) }\lesssim \sqrt{\int
\left\vert T_{\sigma }\left( \mathbf{1}_{\left[ 0,1\right) ^{n-1}\times %
\left[ r,r^{\ast }\right) }\left( y\right) y_{i}\right) \left( x\right)
\right\vert ^{2}d\omega \left( x\right) }+\mathfrak{FT}_{T}^{\left( 1\right)
}\ \left\vert \left[ 0,1\right) ^{n}\right\vert _{\sigma }\ .
\label{i less than n}
\end{equation}%
For the monomial $y_{n}^{2}$ we use the identity 
\begin{equation*}
\mathbf{1}_{\left[ 0,1\right) ^{n}}\left( y\right) y_{n}^{2}=\int_{0}^{1}%
\mathbf{1}_{\left[ 0,1\right) ^{n-1}\times \left[ r,1\right) }\left(
y\right) 2\left( y_{n}-r\right) dr,\ \ \ \ \ \text{for all }y\in \mathbb{R}%
^{n},
\end{equation*}%
to obtain%
\begin{equation*}
\left\Vert T_{\sigma }\left[ \mathbf{1}_{R_{r}}\left( y\right) y_{n}^{2}%
\right] \right\Vert _{L^{2}\left( \omega \right) }\lesssim \sqrt{\int
\left\vert T_{\sigma }\left( \mathbf{1}_{\left[ 0,1\right) ^{n-1}\times %
\left[ r,r^{\ast }\right) }\left( y\right) \left( y_{n}-r\right) \right)
\left( x\right) \right\vert ^{2}d\omega \left( x\right) }+\mathfrak{FT}%
_{T}^{\left( 1\right) }\ \left\vert \left[ 0,1\right) ^{n}\right\vert
_{\sigma }\ .
\end{equation*}%
Then in either case, integrating in $r$, using the simple probability
argument above, and finally using the appropriate analogues of these
estimates for all cubes $Q\in \mathcal{P}^{n}$ in place of $\left[
0,1\right) ^{n}$, we obtain 
\begin{equation*}
\mathfrak{FT}_{T}^{\left( 2\right) }\left( \sigma ,\omega \right) \leq
C_{m,1}\mathfrak{FT}_{T}^{\left( 1\right) }\left( \sigma ,\omega \right)
+\varepsilon \mathfrak{N}_{T}\left( \sigma ,\omega \right) \mathfrak{.}
\end{equation*}

Continuing in this way, using the identity%
\begin{equation*}
\mathbf{1}_{\left[ 0,1\right) ^{n}}\left( y\right) y^{\beta }=\int_{0}^{1}%
\mathbf{1}_{\left[ 0,1\right) ^{n-1}\times \left[ r,1\right) }\left(
y\right) \left( y_{1}^{\beta _{1}}...y_{n-1}^{\beta _{n-1}}\right) \left(
2\beta _{n}\left( y_{n}-r\right) ^{\beta _{n}-1}\right) dr,\ \ \ \ \ \text{%
for all }y\in \mathbb{R}^{n},
\end{equation*}%
yields the inequality%
\begin{equation*}
\mathfrak{FT}_{T}^{\left( \kappa \right) }\left( \sigma ,\omega \right) \leq
C_{m,\kappa -1}\mathfrak{FT}_{T}^{\left( \kappa -1\right) }\left( \sigma
,\omega \right) +\varepsilon \mathfrak{N}_{T}\left( \sigma ,\omega \right)
,\ \ \ \ \ \kappa \in \mathbb{N}.
\end{equation*}%
Iteration then gives%
\begin{eqnarray*}
\mathfrak{FT}_{T}^{\left( \kappa \right) }\left( \sigma ,\omega \right)
&\leq &\varepsilon \mathfrak{N}_{T}\left( \sigma ,\omega \right)
+C_{m,\kappa -1}\mathfrak{FT}_{T}^{\left( \kappa -1\right) }\left( \sigma
,\omega \right) \\
&\leq &\varepsilon \mathfrak{N}_{T}\left( \sigma ,\omega \right)
+C_{m,\kappa -1}\left\{ \varepsilon \mathfrak{N}_{T}\left( \sigma ,\omega
\right) +C_{m,\kappa -2}\mathfrak{FT}_{T}^{\left( \kappa -2\right) }\left(
\sigma ,\omega \right) \right\} \\
&&\vdots \\
&\leq &\varepsilon \left\{ 1+C_{m,\kappa -1}+C_{m,\kappa -1}C_{m,\kappa
-2}+...+C_{m,\kappa -1}C_{m,\kappa -2}...C_{m,0}\right\} \mathfrak{N}%
_{T}\left( \sigma ,\omega \right) \\
&&+\left\{ C_{m,\kappa -1}C_{m,\kappa -2}+...+C_{m,\kappa -1}C_{m,\kappa
-2}...C_{0m,}\right\} \mathfrak{FT}_{T}\left( \sigma ,\omega \right) \\
&=&\varepsilon A\left( \kappa ,\varepsilon \right) \mathfrak{N}_{T}\left(
\sigma ,\omega \right) +B\left( \kappa ,\varepsilon \right) \mathfrak{FT}%
_{T}\left( \sigma ,\omega \right) ,
\end{eqnarray*}%
where the constants $A\left( \kappa ,\varepsilon \right) $ and $B\left(
\kappa ,\varepsilon \right) $ are independent of the operator norm $%
\mathfrak{N}_{T}\left( \sigma ,\omega \right) $. Here we have taken $%
m\approx \log _{2}\frac{1}{\varepsilon }$. This completes the proof of
Theorem \ref{Tp control by T1}.
\end{proof}

We have already pointed out in dimension $n=1$, the equivalence of full
testing with the usual $1$-testing in the presence of one-tailed Muckenhoupt
conditions. In higher dimensions the same is true for at least doubling
measures. For this we use a quantitative expression of the fact that
doubling measures don't charge the boundaries of cubes \cite[see e.g. 8.6
(b) on page 40]{Ste2}.

\begin{lemma}
\label{boundary doubling}Suppose $\sigma $ is a doubling measure on $\mathbb{%
R}^{n}$ and that $Q\in \mathcal{P}^{n}$. Then for $0<\delta <1$ we have%
\begin{equation*}
\left\vert Q\setminus \left( 1-\delta \right) Q\right\vert _{\sigma }\leq 
\frac{C}{\ln \frac{1}{\delta }}\left\vert Q\right\vert _{\sigma }\ .
\end{equation*}
\end{lemma}

\begin{proof}
Let $\delta =2^{-m}$. Denote by $\mathfrak{C}^{\left( m\right) }\left(
Q\right) $ the set of $m^{th}$ generation dyadic children of $Q$ , so that
each $I\in \mathfrak{C}^{\left( m\right) }\left( Q\right) $ has side length $%
\ell \left( I\right) =2^{-m}\ell \left( Q\right) $, and define the
collections%
\begin{eqnarray*}
\mathfrak{G}^{\left( m\right) }\left( Q\right) &\equiv &\left\{ I\in 
\mathfrak{C}^{\left( m\right) }\left( Q\right) :I\subset Q\text{ and }%
\partial I\cap \partial Q\neq \emptyset \right\} , \\
\mathfrak{H}^{\left( m\right) }\left( Q\right) &\equiv &\left\{ I\in 
\mathfrak{C}^{\left( m\right) }\left( Q\right) :3I\subset Q\text{ and }%
\partial \left( 3I\right) \cap \partial Q\neq \emptyset \right\} .
\end{eqnarray*}%
Then 
\begin{equation*}
Q\setminus \left( 1-\delta \right) Q=\mathfrak{G}^{\left( m\right) }\left(
Q\right) \text{ and }\left( 1-\delta \right) Q=\overset{\cdot }{\dbigcup }%
_{k=2}^{m}\mathfrak{G}^{\left( k\right) }\left( Q\right) .
\end{equation*}%
From the doubling condition we have $\left\vert 3I\right\vert _{\sigma }\leq
D\left\vert I\right\vert _{\sigma }$ for all cubes $I$, and so 
\begin{eqnarray*}
\left\vert \mathfrak{H}^{\left( k\right) }\left( Q\right) \right\vert
_{\sigma } &=&\sum_{I\in \mathfrak{H}^{\left( k\right) }\left( Q\right)
}\left\vert I\right\vert _{\sigma }\geq \sum_{I\in \mathfrak{H}^{\left(
k\right) }\left( Q\right) }\frac{1}{D}\left\vert 3I\right\vert _{\sigma }=%
\frac{1}{D}\int \left( \sum_{I\in \mathfrak{H}^{\left( k\right) }\left(
Q\right) }\mathbf{1}_{3I}\right) d\sigma \\
&\geq &\frac{1}{D}\int \left( \sum_{I\in \mathfrak{G}^{\left( k\right)
}\left( Q\right) }\mathbf{1}_{I}\right) d\sigma =\frac{1}{D}\left\vert 
\mathfrak{G}^{\left( k\right) }\left( Q\right) \right\vert _{\sigma }\geq 
\frac{1}{D}\left\vert \mathfrak{G}^{\left( m\right) }\left( Q\right)
\right\vert _{\sigma }=\frac{1}{D}\left\vert Q\setminus \left( 1-\delta
\right) Q\right\vert _{\sigma }\ .
\end{eqnarray*}%
Thus we have%
\begin{equation*}
\left\vert Q\right\vert _{\sigma }\geq \sum_{k=2}^{m}\left\vert \mathfrak{H}%
^{\left( k\right) }\left( Q\right) \right\vert _{\sigma }\geq \frac{m-1}{D}%
\left\vert Q\setminus \left( 1-\delta \right) Q\right\vert _{\sigma }\ ,
\end{equation*}%
which proves the lemma.
\end{proof}

\begin{proposition}
\label{full control for doub}Suppose that $\sigma $ and $\omega $ are
locally finite positive Borel measures on $\mathbb{R}^{n}$, and that $\sigma 
$ is doubling. Then for $0<\varepsilon <1$ there is a positive constant $%
C\left( \varepsilon \right) $ such that%
\begin{equation*}
\mathfrak{FT}_{T}\left( \sigma ,\omega \right) \leq \mathfrak{T}_{T}\left(
\sigma ,\omega \right) +C\left( \varepsilon \right) \mathcal{A}_{2}^{\alpha
}\left( \sigma ,\omega \right) +\varepsilon \mathfrak{N}_{T}\left( \sigma
,\omega \right) \ .
\end{equation*}
\end{proposition}

\begin{proof}
Let $\delta >0$ be defined by the equation $\varepsilon =\frac{C}{\ln \frac{1%
}{\delta }}$, i.e. $\delta =e^{-\frac{C}{\varepsilon }}$. Then we write%
\begin{eqnarray*}
\int_{\mathbb{R}^{n}}\left\vert T_{\sigma }\mathbf{1}_{Q}\right\vert
^{2}d\omega &=&\int_{Q}\left\vert T_{\sigma }\mathbf{1}_{Q}\right\vert
^{2}d\omega +\int_{\mathbb{R}^{n}\setminus Q}\left\vert T_{\sigma }\mathbf{1}%
_{\left( 1-\delta \right) Q}+T_{\sigma }\mathbf{1}_{Q\setminus \left(
1-\delta \right) Q}\right\vert ^{2}d\omega \\
&\leq &\mathfrak{T}_{T}\left( \sigma ,\omega \right) ^{2}\left\vert
Q\right\vert _{\sigma }+2\int_{\mathbb{R}^{n}\setminus Q}\left\vert
T_{\sigma }\mathbf{1}_{\left( 1-\delta \right) Q}\right\vert ^{2}d\omega
+2\int_{\mathbb{R}^{n}\setminus Q}\left\vert T_{\sigma }\mathbf{1}%
_{Q\setminus \left( 1-\delta \right) Q}\right\vert ^{2}d\omega \\
&\leq &\mathfrak{T}_{T}\left( \sigma ,\omega \right) ^{2}\left\vert
Q\right\vert _{\sigma }+C\frac{1}{\delta }\mathcal{A}_{2}^{\alpha }\left(
\sigma ,\omega \right) \left\vert Q\right\vert _{\sigma }+2\mathfrak{N}%
_{T}^{2}\left( \sigma ,\omega \right) \left\vert Q\setminus \left( 1-\delta
\right) Q\right\vert _{\sigma }\ .
\end{eqnarray*}%
Now invoke Lemma \ref{boundary doubling} to obtain%
\begin{equation*}
\int_{\mathbb{R}^{n}}\left\vert T_{\sigma }\mathbf{1}_{Q}\right\vert
^{2}d\omega \leq \mathfrak{T}_{T}\left( \sigma ,\omega \right)
^{2}\left\vert Q\right\vert _{\sigma }+C\frac{1}{\delta }\mathcal{A}%
_{2}^{\alpha }\left( \sigma ,\omega \right) \left\vert Q\right\vert _{\sigma
}+\varepsilon \mathfrak{N}_{T}^{2}\left( \sigma ,\omega \right) \left\vert
Q\right\vert _{\sigma }\ ,
\end{equation*}%
with $\varepsilon =\frac{2C}{\ln \frac{1}{\delta }}$.
\end{proof}

\section{A $T1$ theorem for doubling weights when one weight is $A_{\infty }$%
}

The following $T1$ theorem provides a Cube Testing extension of the $T1$
theorem of David and Journ\'{e} \cite{DaJo} to a pair of weights with one
doubling and the other $A_{\infty }$ (and provided the operator is bounded
on unweighted $L^{2}\left( \mathbb{R}^{n}\right) $ when $\alpha =0$).

\begin{theorem}
\label{A infinity theorem}Suppose $0\leq \alpha <n$, and $\kappa _{1},\kappa
_{2}\in \mathbb{N}$ and $0<\delta <1$. Let $T^{\alpha }$ be an $\alpha $%
-fractional Calder\'{o}n-Zygmund singular integral operator on $\mathbb{R}%
^{n}$ with a standard $\left( \kappa _{1}+\delta ,\kappa _{2}+\delta \right) 
$-smooth $\alpha $-fractional kernel $K^{\alpha }$, and when $\alpha =0$,
suppose that $T^{0}$ is bounded on unweighted $L^{2}\left( \mathbb{R}%
^{n}\right) $. Assume that $\sigma $ and $\omega $ are locally finite
positive Borel doubling measures on $\mathbb{R}^{n}$ with doubling exponents 
$\theta _{1}$ and $\theta _{2}$ respectively satisfying $\kappa _{1}>\theta
_{1}+\alpha -n$ and $\kappa _{2}>\theta _{2}+\alpha -n$, and that one of the
measures is an $A_{\infty }$ weight. Set 
\begin{equation*}
T_{\sigma }^{\alpha }f=T^{\alpha }\left( f\sigma \right)
\end{equation*}%
for any smooth truncation of $T^{\alpha }$.\newline
Then the operator $T_{\sigma }^{\alpha }$ is bounded from $L^{2}\left(
\sigma \right) $ to $L^{2}\left( \omega \right) $, i.e. 
\begin{equation}
\left\Vert T_{\sigma }^{\alpha }f\right\Vert _{L^{2}\left( \omega \right)
}\leq \mathfrak{N}_{T^{\alpha }}\left( \sigma ,\omega \right) \left\Vert
f\right\Vert _{L^{2}\left( \sigma \right) },  \label{strong type}
\end{equation}%
uniformly in smooth truncations of $T^{\alpha }$, provided that the
one-tailed fractional conditions (\ref{one-sided}) of Muckenhoupt hold, and
the two dual Cube Testing conditions (\ref{def Kappa polynomial}) hold with $%
\kappa =1$. Moreover we have%
\begin{equation}
\mathfrak{N}_{T^{\alpha }}\left( \sigma ,\omega \right) \leq C\left( \sqrt{%
\mathcal{A}_{2}^{\alpha }\left( \sigma ,\omega \right) +\mathcal{A}%
_{2}^{\alpha }\left( \omega ,\sigma \right) }+\mathfrak{T}_{T^{\alpha
}}\left( \sigma ,\omega \right) +\mathfrak{T}_{\left( T^{\alpha }\right)
^{\ast }}\left( \omega ,\sigma \right) \right) ,  \label{more}
\end{equation}%
where the constant $C$ depends on $C_{CZ}$ in (\ref{sizeandsmoothness'}) and
the doubling parameters $\left( \beta _{1},\gamma _{1}\right) ,\left( \beta
_{2},\gamma _{2}\right) $ of the weights $\sigma $ and $\omega $, as well as
on the $A_{\infty }$ parameters of one of the weights. If $T^{\alpha }$ is
elliptic, and strongly elliptic if $\frac{n}{2}\leq \alpha <n$, the
inequality can be reversed.
\end{theorem}

\begin{proof}
For convenience we take $\kappa =\kappa _{1}=\kappa _{2}$. From Theorem 4 of 
\cite{Saw3} we have the inequality%
\begin{equation*}
\mathfrak{N}_{T^{\alpha }}\left( \sigma ,\omega \right) \leq C\left( \sqrt{%
\mathcal{A}_{2}^{\alpha }\left( \sigma ,\omega \right) +\mathcal{A}%
_{2}^{\alpha }\left( \omega ,\sigma \right) }+\mathfrak{T}_{T^{\alpha
}}^{\left( \kappa \right) }\left( \sigma ,\omega \right) +\mathfrak{T}%
_{\left( T^{\alpha }\right) ^{\ast }}^{\left( \kappa \right) }\left( \omega
,\sigma \right) \right) ,
\end{equation*}%
where the constant $C$ depends on $C_{CZ}$ in (\ref{sizeandsmoothness'}) and
the doubling parameters $\left( \beta _{1},\gamma _{1}\right) ,\left( \beta
_{2},\gamma _{2}\right) $ of the weights $\sigma $ and $\omega $, as well as
on the $A_{\infty }$ parameters of one of the weights. From Theorem \ref{Tp
control by T1} above, we obtain that for every $0<\varepsilon _{1}<1$, there
is a positive constant $C\left( \kappa ,\varepsilon _{1}\right) $ such that 
\begin{equation*}
\mathfrak{FT}_{T^{\alpha }}^{\left( \kappa \right) }\left( \sigma ,\omega
\right) \leq C\left( \kappa ,\varepsilon _{1}\right) \mathfrak{FT}%
_{T^{\alpha }}\left( \sigma ,\omega \right) +\varepsilon _{1}\mathfrak{N}%
_{T^{\alpha }}\left( \sigma ,\omega \right) \ .
\end{equation*}%
Finally from Proposition \ref{full control for doub}, we obtain that for
every $0<\varepsilon _{2}<1$, there is a positive constant $C\left(
\varepsilon _{2}\right) $ such that%
\begin{equation*}
\mathfrak{FT}_{T^{\alpha }}\left( \sigma ,\omega \right) \leq \mathfrak{T}%
_{T^{\alpha }}\left( \sigma ,\omega \right) +C\left( \varepsilon _{2}\right) 
\sqrt{\mathcal{A}_{2}^{\alpha }\left( \sigma ,\omega \right) }+\varepsilon
_{2}\mathfrak{N}_{T^{\alpha }}\left( \sigma ,\omega \right) \ .
\end{equation*}%
Now we drop dependence on $\left( \sigma ,\omega \right) $ to reduce clutter
of notation, and combining inequalities we obtain%
\begin{eqnarray*}
\mathfrak{N}_{T^{\alpha }} &\leq &C\left( \sqrt{\mathcal{A}_{2}^{\alpha }}+%
\mathfrak{T}_{T^{\alpha }}^{\left( \kappa \right) }+\mathfrak{T}_{\left(
T^{\alpha }\right) ^{\ast }}^{\left( \kappa \right) }\right) \leq C\left( 
\sqrt{A_{2}^{\alpha }}+\mathfrak{FT}_{T^{\alpha }}^{\left( \kappa \right) }+%
\mathfrak{FT}_{\left( T^{\alpha }\right) ^{\ast }}^{\left( \kappa \right)
}\right) \\
&\leq &C\left( \sqrt{\mathcal{A}_{2}^{\alpha }}+C\left( \kappa ,\varepsilon
_{1}\right) \mathfrak{FT}_{T^{\alpha }}+\varepsilon _{1}\mathfrak{N}%
_{T^{\alpha }}+C\left( \kappa ,\varepsilon _{1}\right) \mathfrak{FT}_{\left(
T^{\alpha }\right) ^{\ast }}+\varepsilon _{1}\mathfrak{N}_{T^{\alpha
}}\right) \\
&\leq &C\sqrt{\mathcal{A}_{2}^{\alpha }}+CC\left( \kappa ,\varepsilon
_{1}\right) \left\{ \mathfrak{T}_{T^{\alpha }}+\mathfrak{T}_{\left(
T^{\alpha }\right) ^{\ast }}+C\left( \varepsilon _{2}\right) \left( \sqrt{%
\mathcal{A}_{2}^{\alpha }}+\sqrt{\mathcal{A}_{2}^{\alpha ,\ast }}\right)
+2\varepsilon _{2}\mathfrak{N}_{T^{\alpha }}\right\} +2C\varepsilon _{1}%
\mathfrak{N}_{T^{\alpha }} \\
&\leq &CC\left( \kappa ,\varepsilon _{1}\right) \left( \mathfrak{T}%
_{T^{\alpha }}+\mathfrak{T}_{\left( T^{\alpha }\right) ^{\ast }}\right)
+CC\left( \kappa ,\varepsilon _{1}\right) C\left( \varepsilon _{2}\right)
\left( \sqrt{\mathcal{A}_{2}^{\alpha }}+\sqrt{\mathcal{A}_{2}^{\alpha ,\ast }%
}\right) +\left\{ 2CC\left( \kappa ,\varepsilon _{1}\right) \varepsilon
_{2}+2C\varepsilon _{1}\right\} \mathfrak{N}_{T^{\alpha }}\ .
\end{eqnarray*}%
Choose first $\varepsilon _{1}>0$ so that $2C\varepsilon _{1}<\frac{1}{4}$,
and then choose $\varepsilon _{2}>0$ so that $2CC\left( \kappa ,\varepsilon
_{1}\right) \varepsilon _{2}<\frac{1}{4}$. We can then absorb the final term
on the right into the left hand side to obtain 
\begin{equation*}
\mathfrak{N}_{T^{\alpha }}\left( \sigma ,\omega \right) \leq C_{\kappa
}\left( \mathfrak{T}_{T^{\alpha }}\left( \sigma ,\omega \right) +\mathfrak{T}%
_{\left( T^{\alpha }\right) ^{\ast }}\left( \omega ,\sigma \right) +\sqrt{%
\mathcal{A}_{2}^{\alpha }\left( \sigma ,\omega \right) }+\sqrt{\mathcal{A}%
_{2}^{\alpha ,\ast }\left( \sigma ,\omega \right) }\right) ,
\end{equation*}%
for all suitable truncations of $T^{\alpha }$, and where the constant $%
C_{\kappa }$ depends on the doubling constants of the weights, and on the $%
A_{\infty }$ parameters of one of the weights. If $T^{\alpha }$ is elliptic,
and strongly elliptic if $\frac{n}{2}\leq \alpha <n$, then 
\begin{equation*}
\sqrt{\mathcal{A}_{2}^{\alpha }\left( \sigma ,\omega \right) }+\sqrt{%
\mathcal{A}_{2}^{\alpha ,\ast }\left( \sigma ,\omega \right) }\lesssim 
\mathfrak{N}_{T^{\alpha }}\left( \sigma ,\omega \right) ,
\end{equation*}%
by \cite[Lemma 4.1 on page 92.]{SaShUr7}.
\end{proof}

\begin{remark}
If we drop the assumption that one of the weights is $A_{\infty }$,
inequality (\ref{more}) remains true if we include on the right hand side
the Bilinear Indicator Cube Testing constant $\mathcal{BICT}_{T^{\alpha
}}\left( \sigma ,\omega \right) $ from \cite{Saw2}:%
\begin{equation}
\mathcal{BICT}_{T^{\alpha }}\left( \sigma ,\omega \right) \equiv \sup_{Q\in 
\mathcal{P}^{n}}\sup_{E,F\subset Q}\frac{1}{\sqrt{\left\vert Q\right\vert
_{\sigma }\left\vert Q\right\vert _{\omega }}}\left\vert \int_{F}T_{\sigma
}^{\alpha }\left( \mathbf{1}_{E}\right) \omega \right\vert <\infty ,
\label{def ind WBP}
\end{equation}%
where the second supremum is taken over all compact sets $E$ and $F$
contained in a cube $Q$.
\end{remark}

\subsection{Optimal cancellation conditions}

Using Theorem \ref{A infinity theorem}, we can now obtain a $T1$ version of
Theorem 5 in \cite{Saw3}. For $0\leq \alpha <n$, let $T^{\alpha }$ be a
continuous linear map from rapidly decreasing smooth test functions $%
\mathcal{S}$ to tempered distributions in $\mathcal{S}^{\prime }$, to which
is associated a kernel $K^{\alpha }\left( x,y\right) $, defined when $x\neq
y $, that satisfies the inequalities,%
\begin{equation}
\left\vert \partial _{x}^{\beta }\partial _{y}^{\gamma }K^{\alpha }\left(
x,y\right) \right\vert \leq A_{\alpha ,\beta ,\gamma ,n}\left\vert
x-y\right\vert ^{\alpha -n-\left\vert \beta \right\vert -\left\vert \gamma
\right\vert },\ \ \ \ \ \text{for all multiindices }\beta ,\gamma ;
\label{diff ineq}
\end{equation}%
such kernels are called \emph{smooth} $\alpha $-fractional Calder\'{o}%
n-Zygmund kernels on $\mathbb{R}^{n}$. An operator $T^{\alpha }$ is \emph{%
associated} with a kernel $K^{\alpha }$ if, whenever $f\in \mathcal{S}$ has
compact support, the tempered distribution $T^{\alpha }f$ can be identified,
in the complement of the support, with the function obtained by integration
with respect to the kernel, i.e.%
\begin{equation}
T^{\alpha }f\left( x\right) \equiv \int K^{\alpha }\left( x,y\right) f\left(
y\right) d\sigma \left( y\right) ,\ \ \ \ \ \text{for }x\in \mathbb{R}%
^{n}\setminus \func{Supp}f.  \label{identify}
\end{equation}%
The characterization in terms of (\ref{can cond}) in the next theorem is
identical to that in Stein \cite[Theorem 4 on page 306]{Ste2}, except that
the doubling measures $\sigma $ and $\omega $ appear here in place of
Legesgue measure in \cite{Ste2}.

\begin{theorem}
\label{Stein extension}Let $0\leq \alpha <n$. Suppose that $\sigma $ and $%
\omega $ are locally finite positive Borel doubling measures on $\mathbb{R}%
^{n}$. Suppose also that the measure pair $\left( \sigma ,\omega \right) $
satisfies the one-tailed Muckenhoupt conditions in (\ref{one-sided}), and
that in addition, one of the measures is an $A_{\infty }$ weight. Suppose
finally that $K^{\alpha }\left( x,y\right) $ is a smooth $\alpha $%
-fractional Calder\'{o}n-Zygmund kernel on $\mathbb{R}^{n}$. In the case $%
\alpha =0$, we also assume there is $T^{0}$ associated with the kernel $K^{0}
$ that is bounded on unweighted $L^{2}\left( \mathbb{R}^{n}\right) $.\newline
Then there exists a bounded operator $T^{\alpha }:L^{2}\left( \sigma \right)
\rightarrow L^{2}\left( \omega \right) $, that is associated with the kernel 
$K^{\alpha }$ in the sense that (\ref{identify}) holds, \emph{if and only if}
there is a positive constant $\mathfrak{A}_{K^{\alpha }}\left( \sigma
,\omega \right) $ so that%
\begin{eqnarray}
&&\int_{\left\vert x-x_{0}\right\vert <N}\left\vert \int_{\varepsilon
<\left\vert x-y\right\vert <N}K^{\alpha }\left( x,y\right) d\sigma \left(
y\right) \right\vert ^{2}d\omega \left( x\right) \leq \mathfrak{A}%
_{K^{\alpha }}\left( \sigma ,\omega \right) \ \int_{\left\vert
x_{0}-y\right\vert <N}d\sigma \left( y\right) ,  \label{can cond} \\
&&\ \ \ \ \ \ \ \ \ \ \ \ \ \ \ \ \ \ \ \ \ \ \ \ \ \text{for all }%
0<\varepsilon <N\text{ and }x_{0}\in \mathbb{R}^{n},  \notag
\end{eqnarray}%
along with a similar inequality with constant $\mathfrak{A}_{K^{\alpha ,\ast
}}\left( \omega ,\sigma \right) $, in which the measures $\sigma $ and $%
\omega $ are interchanged and $K^{\alpha }\left( x,y\right) $ is replaced by 
$K^{\alpha ,\ast }\left( x,y\right) =K^{\alpha }\left( y,x\right) $.
Moreover, if such $T^{\alpha }$ has minimal norm, then 
\begin{equation}
\left\Vert T^{\alpha }\right\Vert _{L^{2}\left( \sigma \right) \rightarrow
L^{2}\left( \omega \right) }\lesssim \mathfrak{A}_{K^{\alpha }}\left( \sigma
,\omega \right) +\mathfrak{A}_{K^{\alpha ,\ast }}\left( \omega ,\sigma
\right) +\sqrt{\mathcal{A}_{2}^{\alpha }\left( \sigma ,\omega \right) +%
\mathcal{A}_{2}^{\alpha }\left( \omega ,\sigma \right) },  \label{can char}
\end{equation}%
with implied constant depending on $C_{CZ}$, the doubling constants of the
weights, and the $A_{\infty }$ parameters of the $A_{\infty }$ weight. If $%
T^{\alpha }$ is elliptic, and strongly elliptic if $\frac{n}{2}\leq \alpha <n
$, the inequality can be reversed.
\end{theorem}

\begin{proof}
The proof follows almost verbatim that of Theorem 5 in \cite{Saw2}, but
using Theorem \ref{A infinity theorem} above instead of the $Tp$ theorem
with Bilinear Indicator/Cube Testing in \cite{Saw2}, and which thus
eliminates both the polynomials and the indicator of a compact set $E$ that
appear in the characterization in Theorem 5 of \cite{Saw2}. The
straightforward verification of the details is left to the reader.
\end{proof}

\begin{description}
\item[Concluding comments] The $T1$ theorem here is proved for general CZ
operators, and thus in the absence of any special positivity properties of
the CZ kernels $K^{\alpha }$. As a consequence there is no \emph{catalyst}
available to enable control of the difficult `far below' and `stopping'
terms by `goodness' of cubes in the NTV bilinear Haar decomposition (see
e.g. \cite{NTV4}). In the case of the Hilbert transform, the positivity of
the derivative of the convolution kernel $\frac{1}{x}$ permits the
derivation of a strong catalyst, namely the energy condition, from the
testing and Muckenhoupt conditions (see e.g. \cite{LaSaShUr3}). But the lack
of a suitable catalyst for general CZ operators, see \cite{SaShUr11} and 
\cite{Saw1} for negative results, limits us to using the weighted Alpert
wavelets in \cite{RaSaWi} with doubling measures having one weight in $%
A_{\infty }$. It is an intriguing open question whether or not these
restrictions on the weight pair can be removed. There is no known example of
the failure of a $T1$ theorem for fractional CZ operators.
\end{description}

\end{document}